\documentclass[11pt]{article}

\usepackage{tikz}
\usepackage{subfigure}
\usepackage[english]{babel}
\usepackage{graphicx}

\usepackage[center]{caption2}
\usepackage{amsfonts,amssymb,amsmath,latexsym,amsthm}
\usepackage{multirow}

\def\dist{\mbox{dist}}

\newtheorem{myDef}{Definition}
\newtheorem{thm}{Theorem}[section]
\newtheorem{cor}[thm]{Corollary}
\newtheorem{lem}[thm]{Lemma}

\newtheorem{prob}[thm]{Problem}

\newtheorem{obs}[thm]{Observation}

\newcommand{\G}{\mathcal{G}}
\newcommand{\T}{\mathcal{T}}

\begin{document}

\title{Faber–Krahn type inequality for supertrees}
\author{Hongyu Wang$^a$, \quad Xinmin Hou$^{a,b,c}$\footnote{Email: xmhou@ustc.edu.cn (X. Hou), why71@mail.ustc.edu.cn (H. Wang)}\\
\small $^{a}$ School of Mathematical Sciences\\
\small University of Science and Technology of China, Hefei, Anhui 230026, China.\\
\small  $^{b}$ CAS Key Laboratory of Wu Wen-Tsun Mathematics\\
\small University of Science and Technology of China, Hefei, Anhui 230026, China.\\
\small$^c$ Hefei National Laboratory\\
\small University of Science and Technology of China, Hefei 230088, Anhui, China
}

\date{}

\maketitle

\begin{abstract}
The Faber–Krahn inequality states that the first Dirichlet eigenvalue among all bounded domains is no less than a Euclidean ball with the same volume in $\mathbb{R}^n$~\cite{Chavel FB}. Bıyıkoğlu and Leydold~(J. Comb. Theory, Ser. B., 2007) demonstrated that the Faber–Krahn inequality also holds for the class of trees with boundary with the same degree sequence and characterized the unique extremal tree. Bıyıkoğlu and Leydold (2007) also posed a question as follows: Give a characterization of all graphs in a given class $\mathcal{C}$ with the Faber–Krahn property.
In this paper, we address this question specifically for $k$-uniform supertrees with boundary. 
We introduce a spiral-like ordering (SLO-ordering) of vertices for supertrees, an extension of the SLO-ordering for trees initially proposed by  Pruss [ Duke Math. J., 1998], and prove that the SLO-supertree has the Faber-Krahn property among all supertrees with a given degree sequence. Furthermore, among degree sequences that have a minimum degree $d$ for interior vertices, the SLO-supertree with degree sequence $(d,\ldots,d, d', 1, \dots, 1)$ possesses  the Faber-Krahn property. 
\end{abstract}

\section{Introduction}

The Faber–Krahn inequality is a well-known result on the Riemannian manifolds, which states that the first Dirichlet eigenvalue among all bounded domains is no less than that of a Euclidean ball with the same volume in $\mathbb{R}^n$~\cite{Chavel FB}. This principle was initially conjectured by Lord Rayleigh and later independently proven by Faber and Krahn.
The inequality has been extended and refined in various ways, including sharp quantitative forms~\cite{BPV15} and applications to different classes of domains and operators~\cite{CCLP24}, and  even to graphs and discrete structures. 
Friedman~\cite{Friedman geometric aspects} introduced the concept of a `graph with boundary' and formulated the Dirichlet and Neumann eigenvalue problems for graphs. 

A graph with boundary $G=\left(V^{\circ}\cup\partial V, E^{\circ}\cup\partial E\right)$ consists of a set of interior vertices $V^\circ$, boundary vertices $\partial V$, and boundary edges $\partial E$ that connect interior vertices to boundary vertices, along with interior edges $E^\circ$~\cite{FC Spectral Graph, Friedman geometric aspects}. 
Leydold~\cite{Leydold FK regular trees1, Leydold FK regular trees2} proved the Faber–Krahn type inequality for regular trees and gave a complete characterization of extremal trees. Pruss~\cite{Pruss FK regular trees} also studied the Faber–Krahn inequality for regular trees and relative questions. Bıyıkoğlu and Leydold~\cite{Biyikoğlu Leydold FK trees} demonstrated that the Faber–Krahn inequality also holds for the class of trees with the same degree sequence. The unique extremal tree has a spiral-like ordering, i.e., it is a ball approximation.
Zhang, Zhang, and Zhang~\cite{FK unicyclic} extended the previous results to the realm of unicyclic graphs. They established that the Faber-Krahn type inequality holds for unicyclic graphs with a specified graphic unicyclic degree sequence, subject to minor conditions. 
Bıyıkoğlu and Leydold also proposed the following question in~\cite{Biyikoğlu Leydold FK trees}.

\begin{prob}[\cite{Biyikoğlu Leydold FK trees}]\label{PROB: p1}
Give a characterization of all graphs in a given class $\mathcal{C}$ with the Faber–Krahn property, i.e., characterize those graphs in $\mathcal{C}$ which have minimal first Dirichlet eigenvalue for a given
 “volume".
\end{prob}

In this article, we explore Problem~\ref{PROB: p1} in the context of hypergraphs.

The rest of this article is arranged as follows. In Section 2, we present some notations and state our main theorems. In Section 3, we discuss foundational concepts and provide several lemmas that support our main theorem.  The proof of the main theorems will be given in Section 4. 

\section{Notations and main results}

A hypergraph $\mathcal{G}$ is defined as a pair $\mathcal{G}=(V, E)$, where $V$ represents  a set of vertices and $E$ denotes a set of non-empty subsets of $V$, known as edges or hyperedges. A hypergraph is considered as $k$-uniform if every  edge in $E$ contains exactly $k$ vertices. 
Let $\mathcal{G}=(V, E)$ be a hypergraph without multiple edges. 
For a vertex $v \in V(\G)$, the \textit{degree} of $v$ in $\G$ is defined as the number of edges that include $v$, represented as $d(v)$. A vertex of degree one is called a \textit{pendent vertex}. The list of degrees of vertices in hypergraph $\G$ is called the \textit{degree sequence} of $\G$. 
For distinct vertices $u_1,\ldots, u_{p+1}$ and distinct edges $e_1, e_2, \ldots, e_p$ of $\G$, a vertex-edge alternative sequence $(u_1, e_1, u_2, e_2, \ldots, u_p, e_p, u_{p+1})$ is a \textit{path} $P$ in $\G$ from $u_1$ to $u_{p+1}$ if $u_i, u_{i+1} \in e_i$ for $i=1,2, \ldots, p$. If $u_1=u_{p+1}$, then the path $P$ is called a \textit{cycle}. 
A hypergraph $\G$ is \textit{connected} if there is a path between  any two vertices $u$ and $v$ within $\G$. The \textit{distance} between two vertices $u$ and $v$, denoted as $dist(u, v)$, represents the shortest path length connecting  $u$ and $v$.  Two vertices $u, v\in V$ are considered \textit{adjacent} if they share an edge, denoted by $u\sim v$. Let $N(u)$ be the set of all vertices in $V$ adjacent to $u\in V$.

A connected and acyclic hypergraph is called a \textit{supertree}. This definition was initially proposed  by Li, Shao, and Qi~\cite{k-uniform supertrees}. According to this definition, we have the following observation. 
\begin{obs}\label{OBS: o1}
For each $k$-uniform supertree $\G$, any two edges share at most one common vertex, and $|E(\G)|=\frac{|V(G)-1|}{k-1}$ (Two 3-uniform supertrees have been shown in Figure~\ref{fig:enter-label}).
\end{obs}
In this article, we define a supertree  $\G=\left(V^\circ \cup \partial V, E^\circ \cup \partial E\right)$ with {\em boundary} as a supertree with  boundary vertex set $\partial V$ consisting of all degree one vertices. 

The \textit{adjacency matrix} $A_{\G}=(a_{ij})$ of a hypergraph $\G$, introduced by Banerjee in~\cite{Banerjee spectrum of hypergraphs} (a similar definition also was defined  by Rodr\'iguez in~\cite{Rodriguez Laplacian hypergraphs}) is defined as follows:
$$a_{ij}= \begin{cases}\sum\limits_{{e \in E}\atop{\{i,j\}\subseteq e}} \frac{1}{|e|-1}, & \text { if } i \sim j, \\ {\ \ 0}, & \text { otherwise, }\end{cases}$$
where we write $[n]=\{1,2,\dots,n\}$ for the vertex set of $\G$.
For a vertex $i\in [n]$, it is straightforward  to confirm that the degree $d(i)$ equals the sum of  entries in row $i$ of  the adjacent matrix $A_{\G}$, i.e., $d(i)=\sum_{j=1}^n a_{i j}$. The \textit{Laplacian matrix} of the hypergraph  $\G$ is defined as $L_{\G}=D_{\G}-A_{\G}$, where $D_{\G}$ is the diagonal matrix whose diagonal entries are the degrees of the vertices, specifically  $D_{\G}=\operatorname{diag}\left(d(i)\right)_{i\in [n]}$.
From the definition of $D_{\G}$, we have 
$$(L_{\G})_{ij}= \begin{cases}d(i), & \text { if } i = j, \\ \sum\limits_{e \in E\atop \{i,j\}\subseteq e} \frac{-1}{|e|-1}, & \text { if } i \sim j, \\ 0, & \text { otherwise. }\end{cases}$$
A real number $\lambda$ is defined as the \textit{Dirichlet eigenvalue} of the hypergraph  $\G$ if there exists a non-zero real function $f$ on $V(\G)$ such that:
$$
\begin{cases}L_{\G} f(u)=\lambda f(u) & u \in V^\circ, \\ f(u)=0 & u \in \partial V .\end{cases}
$$
The function $f$ is referred to as an eigenfunction corresponding to $\lambda$.
The associate \textit{Rayleigh quotient} $\mathcal{R}_{L_{\mathcal{G}}}(f)$ of a function $f: V \rightarrow \mathbb{R}$ is defined as
$$
\mathcal{R}_{L_{\mathcal{G}}}(f)=\frac{\left\langle L_{\mathcal{G}} f, f\right\rangle}{\langle f, f\rangle}=\frac{\sum_{i\sim j} \sum_{e \in E\atop \{i, j\} \subseteq e} \frac{1}{|e|-1}(f(i)-f(j))^2}{\sum_{i \in V} f(i)^2}.
$$
Specifically, if $\G$ is a $k$-uniform supertree then any two vertices can be contained in at most one edge. Therefore, for a  $k$-uniform supertree $\G$,
$$
\mathcal{R}_{L_{\mathcal{G}}}(f)=\frac{\left\langle L_{\mathcal{G}} f, f\right\rangle}{\langle f, f\rangle}=\frac{\sum_{i\sim j}\frac{1}{k-1}(f(i)-f(j))^2}{\sum_{i \in V} f(i)^2}.
$$
Let $\lambda(\G)$ be the first Dirichlet eigenvalue of $\G$. Then 
we have
$$
\lambda(\G)=\min_{f \in S} \mathcal{R}_{L_{\G}}(f)=\min_{f \in S} \frac{\langle L_{\G} f, f\rangle}{\langle f, f\rangle},
$$
where $S$ is the set of all non-zero real functions on $V(\G)$ with the constraint $\left.f\right|_{\partial V}=0$. Furthermore, equality holds  if and only if $f$ is an eigenfunction.

\begin{myDef}
    A hypergraph with boundary has the Faber-Krahn property if it has minimal first Dirichlet eigenvalue among all hypergraphs with the same 'volume' in a particular class.
\end{myDef}

In this article, the volume is referred to a given supertree degree sequence $\pi=(d_1,\ldots,d_n)$, or a given interior vertex number and a minimum degree of interior vertices. 
Let
$$\mathcal{T}_{\pi}=\{\G\,|\, \G\text{ is a $k$-uniform supertree with }\pi\text{ as its degree sequence}\},$$
$$\mathcal{T}^{n,n_0}=\{\G\, |\, \G\text{ is a $k$-uniform supertree with }|V|=n\text{ and }|V^\circ|=n_0\},$$
and 
$$\mathcal{T}^{n,n_0}_d=\{\G\in \mathcal{T}^{n,n_0}\, |\, d(v)\geq d\text{ for all }v\in V^\circ\}.$$

For a supertree $\G$ with root $v_0$, the \textit{height} $h(v)$ of a vertex $v$ is defined by $h(v) = \dist(v, v_0)$. 
We call a vertex $u$ a {\em child} of vertex $v$ if $u$ and $v$ are adjacent in $\G$ and $h(u)=h(v)+1$, and $v$ is called the {\em parent} of $u$ (the parent of $u$ is unique in a supertree).

\begin{myDef}[SLO-ordering supertree] Let $\mathcal{G}=(V^\circ \cup \partial V, E^\circ \cup \partial E)$ be a $k$-uniform supertree with root $v_0$. A well-ordering $\prec$ of the vertices is called a spiral-like ordering (SLO-ordering for short) if the following holds for all vertices in $\mathcal{G}$:
\begin{itemize}
    \item[(S1)] $u \prec v$ implies $h(u) \leq h(v)$. 
    \item[(S2)] Suppose $u$ and $v$ are the parents of $u_1$ and $v_1$, respectively. If $u \prec v$, then $u_1 \prec v_1$.
    \item[(S3)] If $u \prec v$ and $u\in \partial V$, then $v\in \partial V$.
    \item[(S4)] Suppose $u_1 \prec u_2 \prec \cdots \prec u_k$ for an edge $e=\left\{u_1, u_2, \ldots, u_k\right\} \in E(\G)$, then there exists no vertex $v \in V(\G) \backslash e$ such that $u_i \prec v \prec u_{i+1}$ for $2 \leq i \leq k-1$.
\item[(S5)] For any $u,v\in V^\circ$, $u \prec v$ implies $d(u) \leq d(v)$.
\end{itemize}
\end{myDef}
\noindent{\bf Remark A:} The rule (S4) implies that  for an edge $e=\left\{u_1, u_2, \ldots, u_k\right\}\in E(\G)$ with $u_1 \prec u_2 \prec \cdots \prec u_k$, 
the vertices in $e$ except $u_1$ are consecutive in this SLO-ordering.

\begin{figure}
    \centering
    \includegraphics[width=14cm]{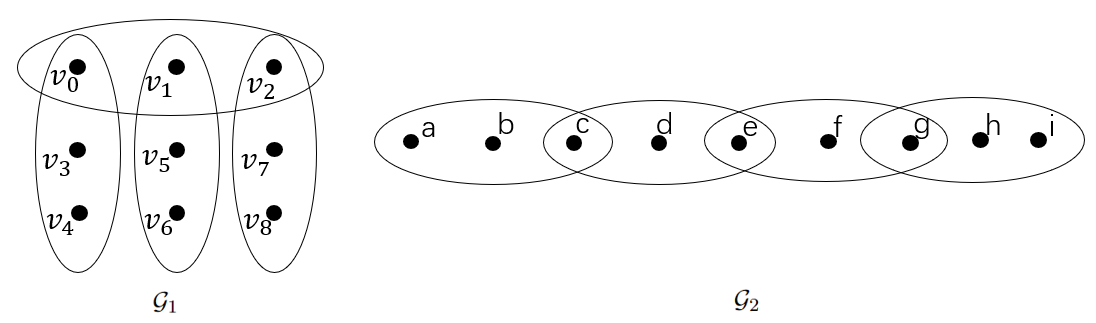}
    \caption{$\G_1$ is an SLO-ordering supertree, however, $\G_2$ is not.}
    \label{fig:enter-label}
\end{figure}
The two 3-uniform supertrees  $\G_1$ and $\G_2$ as shown in Figure 1 have the same degree sequence $(3,3,3,1,1,1,1,1,1)$. The supertree $\G_1$ has an SLO-ordering $v_0\prec v_1\prec v_2\prec\dots\prec v_8$. However, we cannot find an SLO-ordering in $\G_2$. Suppose to the contrary that $\G_2$ has an SLO-ordering $v_0\prec v_1\prec v_2\prec\dots\prec v_8$. If we choose $c$ as the root $v_0$, then vertices $a,b,d,e$ must be ordered before vertex $g$ according to (S1). However, this leads to a contradiction to (S3) (as boundary vertices must be ordered after interior vertices according to (S3)). The same discussion for the case $v_0=g$. Finally, if we choose $v_0=e$, then according to (S1) and (S3), we must have  $\{v_1, v_2\}=\{c,g\}$ and $\{v_3, v_4\}=\{d, f\}$. Without loss of generality, assume $v_1=c$ and $v_2=g$. Then for edge $e_0=\{e, c, d\}$, we have $c\prec g\prec d$, a contradict to (S4) again.

We refer to supertrees that have SLO-ordering of their vertices as \textit{SLO-supertrees}.
The following are our main results.

\begin{thm}\label{mainthm}
    A $k$-uniform supertree $\G$ with degree sequence $\pi$ has the Faber-Krahn property within the class $\mathcal{T}_{\pi}$ if and only if $\G$ has an SLO-ordering. Furthermore, $\G$ is uniquely determined  up to isomorphism.
\end{thm}

\begin{thm}\label{mainthm2}
    A $k$-uniform supertree $\G$ exhibits the Faber-Krahn property within the class $\mathcal{T}_d^{n,n_0}$ if and only if it is an SLO-supertree characterized by having at most one interior vertex with degree $d'$ greater than $d$, while all other interior vertices possess degree $d$. Furthermore, this supertree $\G$ is unique up to isomorphism.
\end{thm}

\noindent{\bf Remark B:} Theorem~\ref{mainthm} demonstrates that the SLO-supertree has the Faber-Krahn property among all supertrees with a given degree sequence $\pi$. Additionally, among degree sequences that have a minimum degree $d$ for interior vertices, the SLO-supertree with degree sequence $(d,\ldots,d, d', 1, \dots, 1)$ possesses  the Faber-Krahn property.

The rest of this article is arranged as follows. In Section 3, we present some preliminaries and lemmas. The proofs of Theorems~\ref{mainthm} and \ref{mainthm2} will be provided in Sections 4 and 5, respectively. 

\section{Preliminaries and lemmas}
First, we give the following property of the first Dirichlet eigenvalue of a supertree with boundary.
\begin{lem}\label{simple}
     Let $\mathcal{G}=\left(V^\circ \cup \partial V, E^\circ \cup \partial E\right)$ be a $k$-uniform supertree with boundary. Then
     \begin{itemize}
         \item[(1)] The first Dirichlet eigenvalue $\lambda(G)$ is a positive simple eigenvalue.
         \item[(2)] An eigenfunction $f$ of the eigenvalue $\lambda(G)$ is either positive or negative on all interior vertices of $G$.
     \end{itemize}
\end{lem}
\begin{proof}
Let $f$ be an eigenfunction corresponding to $\lambda(\G)$. Then
\begin{equation}\label{Eq: e1}
\begin{aligned}
     &\left<L_{\mathcal{G}}f,f\right>-\left<L_{\mathcal{G}}|f|,|f|\right>\\
     &\,\,= \sum_{\{u_i, u_j\}\subseteq e\atop e \in E}\frac{1}{k-1}(f(u_i)-f(u_j))^2-\sum_{\{u_i, u_j\}\subseteq e\atop e \in E}\frac{1}{k-1}(|f(u_i)|-|f(u_j)|)^2\\
     &\,\, = \sum_{\{u_i, u_j\}\subseteq e\atop e \in E}\frac{2}{k-1}(|f(u_i)||f(u_j)|-f(u_i)f(u_j))\\
     &\,\, \geq  0
\end{aligned}
\end{equation}
Together with $\mathcal{R}_{\mathcal{G}}(f)\leq\mathcal{R}_{\mathcal{G}}(|f|)$, we have $\left<L_{\mathcal{G}}f,f\right>=\left<L_{\mathcal{G}}|f|,|f|\right>$. By (\ref{Eq: e1}), we have  $f\geq 0$ or $f\leq 0$ on all interior vertices of $\G$.

Now suppose $f\ge 0$, and however, $f=0$ on some interior vertices. Since $\G$ is connected and $f\neq 0$, we can find two interior vertices   $\{u,v\}\subseteq e \in E$ with $f(u)=0$ and $f(v)>0$. Let $\tilde{f}=f+\epsilon \chi_u$, where $\epsilon>0$ is sufficiently small and $\chi_u$ is the characteristic function of $u$.
Then 
\begin{equation*}
\begin{aligned}
    \left<L_{\mathcal{G}}f,f\right>-\left<L_{\mathcal{G}}\tilde{f},\tilde{f}\right>
     =& \sum_{\{u, v_i\}\subseteq e\atop e \in E}\frac{1}{k-1}((0-f(v_i))^2-(\epsilon-f(v_i))^2)\\
     =& \sum_{\{u, v_i\}\subseteq e\atop e \in E}\frac{2\epsilon}{k-1}(f(v_i)-\epsilon)\\
     \geq& \frac{2\epsilon}{k-1}f(v)-\frac{2\epsilon}{k-1}(k-1)\epsilon\\
     =& \frac{2\epsilon}{k-1}(f(v)-(k-1)\epsilon)>0,
\end{aligned}
\end{equation*}
a contradiction to the smallest of $f$.
Therefore, $f$ is either positive or negative on all interior vertices of $\G$.

Since there are no two functions with a strict sign on interior vertices that can be orthogonal, it follows that the first eigenvalue has multiplicity one.
\end{proof}

 According to Lemma~\ref{simple},  the first Dirichlet eigenvalue of $\G$ has multiplicity one. Therefore, there exists only one eigenfunction $f$ of $\lambda(G)$ that satisfies $f(v)>0$ for $v \in V^\circ, f(u)=0$ for $u \in \partial V$ and $\|f\|=\left(\sum_{u\in V}f(u)^2\right)^{\frac{1}{2}}=1$. We call $f$ the {\it first Dirichlet eigenfunction} of $\G$.

\begin{myDef}[Switching-operation] 
For two edges $e_1$ and $e_2$ in a $k$-uniform hypergraph $\mathcal{G}$, let $U_1\subseteq e_1$ and $V_1\subseteq e_2$ with $|U_1|=|V_1|$. We define the {\em switching operation $e_1 \underset{\scriptscriptstyle V_1}{\stackrel{\scriptscriptstyle U_1}{\rightleftharpoons}} e_2$} as follows: replace $e_1$ with $e'_1=(e\backslash U_1)\cup V_1$ and replace $e_2$ with $e_2'=(e_2\backslash V_1)\cup U_1$.

   
\end{myDef}
 Let  $\G'=\G\left(e_1\underset{\scriptscriptstyle  V_1}{\stackrel{\scriptscriptstyle  U_1}{\rightleftharpoons}} e_2\right)$ be the hypergraph obtained from $\G$ by the switching-operation $e_1 \underset{\scriptscriptstyle  V_1}{\stackrel{\scriptscriptstyle  U_1}{\rightleftharpoons}} e_2$, that is $V(\G')=V(\G)$ and
$$E\left(\G^{\prime}\right)=(E(\G) \backslash\{e_1, e_2\}) \cup\left\{e_1^{\prime}, e_2^{\prime}\right\}.$$

\begin{lem}\label{switching}
Let $\mathcal{G}=(V^\circ \cup \partial V, E^\circ \cup \partial E)$ be a $k$-uniform supertree and let $f$ be the first Dirichlet eigenfunction of $\G$. Suppose $e_1$ and $e_2$ are two edges in $\mathcal{G}$. Let $U_1\subseteq e_1$ and $V_1\subseteq e_2$ with $|U_1|=|V_1|$. Let $\G'=\G\left(e_1 \underset{\scriptscriptstyle  V_1}{\stackrel{\scriptscriptstyle  U_1}{\rightleftharpoons}} e_2\right)$. 
If $\sum_{u\in U_1} f(u)\geq \sum_{v\in V_1} f(v)$ and $\sum_{u\in e_1\setminus U_1}f(u)\leq \sum_{v\in e_2\setminus V_1} f(v)$, then $\lambda(\mathcal{\G}')\leq \lambda(\mathcal{\G}).$ 
Moreover, if one of the inequalities $\sum_{u\in U_1} f(u)\geq \sum_{v\in V_1} f(v)$ and $\sum_{u\in e_1\setminus U_1}f(u)\leq \sum_{v\in e_2\setminus V_1} f(v)$ is strict, then $\lambda(\mathcal{\G}')< \lambda(\mathcal{\G})$.
\end{lem}
\begin{proof}
Let $e_1=\left\{u_1, \ldots, u_k\right\}$ and $e_2=\left\{v_1, \ldots, v_k\right\}$ and let $U_1=\{u_1,\ldots ,u_l\}$ and $V_1=\{v_1,\ldots,v_l\},(1\leq l\leq k)$. 
Then $\sum_{i=1}^lf(v_i)\le\sum_{i=1}^lf(u_i)$ and $\sum_{j=l+1}^kf(v_j)\ge \sum_{j=l+1}^kf(u_j)$. Thus
\begin{equation*}
\begin{aligned}
     &\left<L_{\mathcal{G}'}f,f\right>-\left<L_{\mathcal{G}}f,f\right>\\
     &\ \ = \sum_{1\leq i\leq l\atop l+1\leq j\leq k}\frac{1}{k-1}(f(u_i)-f(v_j))^2+\sum_{1\leq i\leq l\atop l+1\leq j\leq k}\frac{1}{k-1}(f(v_i)-f(u_j))^2\\
     &\ \ \ \ \ \ -\sum_{1\leq i\leq l\atop l+1\leq j\leq k}\frac{1}{k-1}(f(u_i)-f(u_j))^2-\sum_{1\leq i\leq l\atop l+1\leq j\leq k}\frac{1}{k-1}(f(v_i)-f(v_j))^2\\
    &\ \ = \frac{1}{k-1}\sum_{1\leq i\leq l\atop l+1\leq j\leq k}2(f(v_j)-f(u_j))(f(v_i)-f(u_i))\\
     &\ \ = \frac{2}{k-1}\left(\sum_{1\leq i\leq l}f(v_i)-\sum_{1\leq i\leq l}f(u_i)\right)\left(\sum_{l+1\leq j\leq k}f(v_j)-\sum_{l+1\leq j\leq k}f(u_j)\right)\\
    &\ \  \leq  0.
\end{aligned}
\end{equation*}
Therefore, $\lambda(\mathcal{G}')\leq\mathcal{R}_{L_{\mathcal{G}'}}(f)\leq\mathcal{R}_{L_\mathcal{G}}(f)=\lambda(\mathcal{G})$.
Moreover, if $\lambda(\mathcal{G}')=\lambda(\mathcal{G})$, then $\mathcal{R}_{L_{\mathcal{G}'}}(f)=\mathcal{R}_{L_\mathcal{G}}(f)$, and thus $f$ is the first eigenfunction of $\lambda(\mathcal{G}')$.

If $\sum_{u\in U_1} f(u) > \sum_{v\in V_1} f(v)$, and we  assume $\lambda(\mathcal{G}')=\lambda(\mathcal{G})$,  then $\mathcal{R}_{L_{\mathcal{G}'}}(f)=\mathcal{R}_{L_\mathcal{G}}(f)$ and $f$ is the first eigenfunction of $\lambda(\mathcal{G}')$ too.  Thus
\begin{equation}
\begin{aligned}
    &0=\lambda(\mathcal{G})f(v_k)-\lambda(\mathcal{G}')f(v_k)=L_\mathcal{G}f(v_k)-L_{\mathcal{G}'}f(v_k)=\frac{1}{k-1}\left(\sum_{i=1}^{l}f(v_i)-\sum_{i=1}^{l}f(u_i)\right)<0,
\end{aligned}
\end{equation}
a contradiction.  Therefore, we have $\lambda(\mathcal{G}')< \lambda(\mathcal{G})$.

If $\sum_{u\in e_1\setminus U_1}f(u)< \sum_{v\in e_2\setminus V_1} f(v)$, and we assume $\lambda(\mathcal{G}')=\lambda(\mathcal{G})$, then $\mathcal{R}_{L_{\mathcal{G}'}}(f)=\mathcal{R}_{L_\mathcal{G}}(f)$ and $f$ is also the first eigenfunction of $\lambda(\mathcal{G}')$.  Thus
\begin{equation}
\begin{aligned}
    &0=\lambda(\mathcal{G})f(v_1)-\lambda(\mathcal{G}')f(v_1)=L_\mathcal{G}f(v_1)-L_{\mathcal{G}'}f(v_1)=\frac{1}{k-1}\left(\sum_{i=l+1}^{k}f(v_i)-\sum_{i=l+1}^{k}f(u_i)\right)>0,
\end{aligned}
\end{equation}
a contradiction again.  Therefore, $\lambda(\mathcal{G}')< \lambda(\mathcal{G})$.

\end{proof}

\begin{myDef}[The shifting-operation] For an edge $e$ with $u\in e$ and a vertex $v$ in a $k$-uniform hypergraph $\mathcal{G}$, we define the {\em shifting-operation $u \xrightarrow{e} v$} as replacing $e$ by $e'=(e\backslash\{u\})\cup\{v\}$.

\end{myDef}
    Let  $\G'=\G(u \xrightarrow{e} v)$ be the  graph resulting from  the shifting-operation  $u \xrightarrow{e} v$ on $\G$. Then we  have $V(\G')=V(\G)$ and $E\left(\G^{\prime}\right)=(E(\G) \backslash \{e\}) \cup\{e'\}.$ For an edge set $F$ with $u\in e$ for every edge $e\in F$ and a vertex $v$ in hypergraph $\mathcal{G}$, we use $u \xrightarrow{F} v$ denote the shifting-operations  $u \xrightarrow{e} v$ for every edge $e\in F$ on $\G$.

\begin{lem}\label{shifting}
 Let $\mathcal{G}=(V^{\circ}\cup\partial V, E^{\circ}\cup\partial E)$ be a $k$-uniform supertree. Suppose $F\subseteq E(\G)$ with $|F|=r$ and every edge in $F$ contains the vertex $u$, and let $v$ be  a vertex in $\mathcal{G}$. Let $\G'=\G(u \xrightarrow{F} v)$. If $f(u)\geq f(v)\geq f(w)$ for all $w\in V(F)\setminus\{u\}$, then $\lambda(\mathcal{G}')\leq \lambda(\mathcal{G}).$ Furthermore, if any of these inequalities are strict, then $\lambda(\mathcal{G}')< \lambda(\mathcal{G}).$

\end{lem}
\begin{proof}
Let $F=\{e_1, e_2, \dots, e_r\}$ and $e_i=\left\{u,u_{i,1} \ldots, u_{i,k-1}\right\}$ for $1\leq i\leq r$. Then 
	
    \begin{equation*}
\begin{aligned}
     &\left<L_{\mathcal{G}'}f,f\right>-\left<L_{\mathcal{G}}f,f\right>\\
    &\,\, = \sum_{i=1}^r\left(\sum_{j=1}^{k-1}\frac{1}{k-1}(f(v)-f(u_{i,j}))^2-\sum_{j=1}^{k-1}\frac{1}{k-1}(f(u)-f(u_{i,j}))^2\right)\\
    &\,\, = \frac{1}{k-1}\sum_{i=1}^r\left(\sum_{j=1}^{k-1}(f(v)-f(u))\left(f(v)+f(u)-2f(u_{i,j})\right)\right)\\
   &\,\,  \leq  0.
\end{aligned}
\end{equation*}
Therefore, $\lambda(\mathcal{G}')\leq\mathcal{R}_{\mathcal{G}'}(f)\leq\mathcal{R}_{\mathcal{G}}(f)=\lambda(\mathcal{G})$.
Moreover, if  $\lambda(\mathcal{G}')=\lambda(\mathcal{G})$ then $\mathcal{R}_{L_{\mathcal{G}'}}(f)=\mathcal{R}_{L_\mathcal{G}}(f)$ and thus $f$ is an eigenfunction of $\lambda(\mathcal{G}')$.

If $f(v)> f(u_{i,j})$ for some $1\leq i \leq r$, $1\leq j \leq k-1$ and $\lambda(\mathcal{G}')=\lambda(\mathcal{G})$, then we have
\begin{equation*}
\begin{aligned}
    0=\lambda(\mathcal{G}')f(v)-\lambda(\mathcal{G})f(v)=L_{\mathcal{G}'}f(v)-L_\mathcal{G}f(v)=\frac{1}{k-1}\sum_{i=1}^{r}\sum_{j=1}^{k-1}(f(v)-f(u_{i,j}))>0,
\end{aligned}
\end{equation*}
a contradiction. 

 If $f(u)>f(u_{i,j})$ for some $1\leq i \leq r$, $1\leq j \leq k-1$ and  $\lambda(\mathcal{G}')=\lambda(\mathcal{G})$, then we have
\begin{equation*}
\begin{aligned}
    0=\lambda(\mathcal{G})f(u)-\lambda(\mathcal{G}')f(u)=L_{\mathcal{G}}f(u)-L_{\mathcal{G}'}f(u)=\frac{1}{k-1}\sum_{i=1}^{r}\sum_{j=1}^{k-1}(f(u)-f(u_{i,j}))>0,
\end{aligned}
\end{equation*}
a contradiction.

Now suppose $f(u)> f(v)$. If $\lambda(\mathcal{G}')=\lambda(\mathcal{G})$, then 
\begin{equation*}
\begin{aligned}
    0=\lambda(\mathcal{G}')f(u_{1,1})-\lambda(\mathcal{G})f(u_{1,1})=L_{\mathcal{G}'}f(u_{1,1})-L_\mathcal{G}f(u_{1,1})=\frac{1}{k-1}(f(u)-f(v))>0,
\end{aligned}
\end{equation*}
a contradiction again.

\end{proof}

\section{Proof of Theorem~\ref{mainthm}}
In the subsequent discussion, let $\G(V^{\circ}\cup\partial V, E^{\circ}\cup\partial E)$ be a $k$-uniform supertree with boundary in the class $\T_{\pi}$. Additionally, assume  $\G$ has the Faber–Krahn property. 
Let $f$ be an eigenfunction corresponding to the first Dirichlet eigenvalue of $\G$.

We now introduce the following method of relabeling the vertices of $V(\G)$, which is first proposed by Bıyıkoğlu and Leydold~\cite{Biyikoğlu Leydold FK trees relabeling} and Zhang~\cite{X.-D. Zhang relabeling} on graphs, later extended to hypergraphs by Xiao, Wang and Lu~\cite{Xiao spec supertree deg seq}.

Take a vertex $v_0$ such that  $f(v_0)=\max_{v \in V(\G)} f(v)$ as the root vertex. 
Recall that the height of $v$ is $h(v)=\dist(v, v_0)$. Let $h=\max_{v \in V(\G)}h(v)$. Let $V_i=\left\{v \mid \dist\left(v, v_0\right)=i\right\}$, $a_i=\left|V_i\right|$, and $b_i=\sum_{v \in V_i} d(v)$ for $0 \leq i \leq h$. 
We now relabel all vertices of $\G$ according to the following rules.

\begin{itemize}
\item[] {\bf Relabeling rules:}
    \item[(1)] Label vertex $v_0$ as $v_{0,1,1}$, designating it as the root of the supertree $\G$.
    \item[(2)] All vertices of $\G$ adjacent to $v_{0,1,1}$ are relabeled as $v_{1,1,1}, \ldots, v_{1,1, k-1}$, $v_{1,2,1}, \ldots, v_{1,2, k-1}$, $\ldots, v_{1, b_0, 1}, \ldots, v_{1, b_0, k-1}$, and these vertices satisfy the following conditions:
    \begin{itemize}
        \item[($a$)] The set $e_{1, i}=\left\{v_{0,1,1}, v_{1, i, 1}, \ldots, v_{1, i, k-1}\right\}$ is an edge of $\G$ for $1 \leq i \leq b_0$.
        \item[($b$)] $\sum_{p=1}^{k-1} f(v_{1, i, p}) \geq \sum_{q=1}^{k-1} f(v_{1, j, q})$ for $1 \leq i<j \leq b_0$.
        \item[($c$)] $f(v_{1, i, p}) \geq f(v_{1, i, q})$ for all $1 \leq i \leq b_0$ and $1 \leq p<q \leq k-1$.
    \end{itemize}
    \item[(3)] Assuming that all vertices of $V_r$ have been relabeled for $1 \leq r \leq h-1$, we now relabel the vertices of $V_{r+1}$ as $v_{r+1,1,1}, v_{r+1,1,2}, \ldots, v_{r+1,1, k-1}$, $v_{r+1,2,1}, \ldots, v_{r+1,2, k-1}$, $\ldots, v_{r+1, b_r-a_r, 1}, \ldots, v_{r+1, b_r-a_r, k-1}$. These vertices satisfy the following conditions:
    \begin{itemize}
        \item[($a^{\prime}$)] There exists an edge $e \in E(\G)$ such that $\left\{v_{r+1, i, 1}, v_{r+1, i, 2}, \ldots, v_{r+1, i, k-1}\right\} \subset e$ for $1 \leq i \leq b_r-a_r$.
        \item[($b^{\prime}$)] $\sum_{p=1}^{k-1} f(v_{r+1, i, p}) \geq \sum_{q=1}^{k-1} f(v_{r+1, j, q})$ for all $1 \leq i<j \leq b_r-a_r$.
        \item[($c^{\prime}$)] If $\sum_{p=1}^{k-1} f(v_{r+1, i, p})=\sum_{q=1}^{k-1} f(v_{r+1, j, q})$ for $1 \leq i<j \leq b_r-a_r$ and the sets $e_{r+1,i}=\left\{v_{r, i_1, p_1}, v_{r+1, i, 1}, \ldots, v_{r+1, i, k-1}\right\}$ and $e_{r+1,j}=\left\{v_{r, j_1, q_1}, v_{r+1, j, 1}, \ldots, v_{r+1, j, k-1}\right\} \in E(\G)$, then $i_1<j_1$ or $i_1=j_1, p_1<q_1$.
        
        \item[($d^{\prime}$)]$f(v_{r+1, i, p}) \geq f(v_{r+1, i, q})$ for all $1 \leq i \leq b_r-a_r$ and $1 \leq p<q \leq k-1$.
    \end{itemize}
\end{itemize}

Define an ordering $\prec$ on $V(\G)$ as follows:
$$
v_{s, i, p} \prec v_{t, j, q} \text { if and only if } (s,i,p)<(t,j,q) \text { in the lexicographic order}.
$$
Then the following holds.



\begin{lem}\label{f(v)<f(u)}
If $v_{s, i, p} \prec v_{t, j, q}$ then 	$f(v_{s, i, p}) \geq f(v_{t, j, q})$.

\end{lem}

\begin{proof}
    By the relabeling rules ($c$) and ($d^\prime$), we have $f(v_{s, i, p}) \geq f(v_{t, j, q})$ for $s=t, i=j, p<q$. We only need to prove that $f(v_{s, i, p}) \geq f(v_{t, j, q})$ holds when $s<t$ or $s=t$ but $i<j$.
    
\noindent\textbf{Case 1:} $s=t$ but $i<j$.
    
We prove by induction on $s$. When $s=t=1$. Suppose that $f(v_{1,i,p_0})<f(v_{1,j,q_0})$ for some $1\le p_0, q_0\le k-1$. By the Relabeling rule ($b$), we have $\sum_{p=1}^{k-1} f(v_{1, i, p}) \geq \sum_{q=1}^{k-1} f(v_{1, j, q})$. Thus $\sum_{p=1}^{k-1} f(v_{1, i, p})-f(v_{1,i,p_0}) > \sum_{q=1}^{k-1} f(v_{1, j, q})-f(v_{1,j,q_0})$. Let $\mathcal{G'}$ be the new hypergraph obtained by the edge-switching $e_{1,i} \underset{\scriptscriptstyle v_{1,j,q_0}}{\stackrel{\scriptscriptstyle v_{1,i,p_0}}{\rightleftharpoons}} e_{1,j}$. Then $\mathcal{G'}\in \mathcal{T}_\pi$. However, $\lambda(\mathcal{G}')< \lambda(\mathcal{G})$ by Lemma~\ref{switching}, this leads to a contradiction. Therefore, $f(v_{1,i,p})\geq f(v_{1,j,q})$ for $i<j$.

Now assume $f(v_{r, i, p})\ge f(v_{r, j, q})$ when $s=t=r>1$. Suppose that $f(v_{r+1, i, p_0}) < f(v_{r+1, j, q_0})$ for some $1\le p_0, q_0\le k-1$. 
By the Relabeling rule ($a^\prime$), there exist $e_{r+1,i}=\left\{v_{r,i_1,i_2}, v_{r+1, i, 1}, \ldots, v_{r+1, i, k-1}\right\}$ and $e_{r+1,j}=\left\{v_{r,j_1,j_2}, v_{r+1, j, 1}, \ldots, v_{r+1, j, k-1}\right\}$ in $\G$ with $v_{r+1, i, p_0}\in e_{r+1,i}$ and $v_{r+1, j, q_0}\in e_{r+1,j}$. By the Relabeling rule ($b'$),  we have $\sum_{p=1}^{k-1} f(v_{r+1, i, p}) \geq \sum_{q=1}^{k-1} f(v_{r+1, j, q})$.
 
If $v_{r,i_1,i_2}=v_{r,j_1,j_2}$, then  $\sum_{p=1}^{k-1} f(v_{r+1, i, p})-f(v_{r+1,i,p_0}) > \sum_{q=1}^{k-1} f(v_{r+1, j, q})-f(v_{r+1,j,q_0})$. Let $\mathcal{G'}$ be the new hypergraph obtained by the edge-switching operation $e_{r+1,i} \underset{v_{r+1,j,q_0}}{\stackrel{v_{r+1,i,p_0}}{\rightleftharpoons}} e_{r+1,j}$. Then $\mathcal{G'}\in \mathcal{T}_\pi$ and according to  Lemma~\ref{switching}, we have $\lambda(\mathcal{G}')< \lambda(\mathcal{G})$. This leads to a contradiction. Therefore, $f(v_{r+1,i,p})\geq f(v_{r+1,j,q})$ when $i<j$.

If $v_{r,i_1,i_2}\prec v_{r,j_1,j_2}$, then $i_1<j_1$ or $i_1=j_1, i_2<j_2$.  By the induction hypothesis, we have  $f(v_{r,i_1,i_2})\geq f(v_{r,j_1,j_2})$.
Thus $f(v_{r,i_1,i_2})+\sum_{p=1}^{k-1} f(v_{r+1, i, p})-f(v_{r+1,i,p_0}) > f(v_{r,j_1,j_2})+\sum_{q=1}^{k-1} f(v_{r+1, j, q})-f(v_{r+1,j,q_0})$. Let $\mathcal{G'}$ be the new hypergraph obtained by the edge-switching operation $e_{r+1,i} \underset{v_{r+1,j,q_0}}{\stackrel{v_{r+1,i,p_0}}{\rightleftharpoons}} e_{r+1,j}$. Then $\mathcal{G'}\in \mathcal{T}_\pi$ and according to  Lemma~\ref{switching}, we have $\lambda(\mathcal{G}')< \lambda(\mathcal{G})$. This leads to a contradiction again. Therefore, we have  $f(v_{r+1,i,p})\geq f(v_{r+1,j,q})$ for $i<j$.

If $v_{r,j_1,j_2}\prec v_{r,i_1,i_2}$, then $j_1<i_1$ or $j_1=i_1$, $j_2<i_2$.
By the induction hypothesis, we have $f(v_{r,i_1,i_2})\leq f(v_{r,j_1,j_2})$. Thus $\sum_{p=1}^{k-1} f(v_{r+1, i, p})-f(v_{r+1,i,p_0}) > \sum_{q=1}^{k-1} f(v_{r+1, j, q})-f(v_{r+1,j,q_0})$, and $f(v_{r,i_1,i_2})+f(v_{r+1,i,p_0}) < f(v_{r,j_1,j_2})+f(v_{r+1,j,q_0})$. 
Set $U_1=\{v_{r+1,i,p_0}, v_{r,i_1,i_2}\}$ and $V_1=\{v_{r+1,j,q_0}, v_{r,j_1,j_2}\}$.
Let $\mathcal{G'}$ be the new hypergraph obtained by the edge-switching operation $e_{r+1,i} \underset{\scriptscriptstyle V_1}{\stackrel{\scriptscriptstyle U_1}{\rightleftharpoons}} e_{r+1,j}$. Then $\mathcal{G'}\in \mathcal{T}_\pi$ and according to Lemma~\ref{switching}, $\lambda(\mathcal{G}')< \lambda(\mathcal{G})$, This leads to a contradiction too. Therefore, $f(v_{r+1,i,p})\geq f(v_{r+1,j,q})$ for $i<j$.      

\noindent\textbf{Case 2:} $s<t$

From Case 1, the Relabeling rules ($c$) and ($d'$), it is sufficient to prove that $f(v_{r, b_{r-1}-a_{r-1}, k-1}) \geq f(v_{r+1,1,1}) \text { for } r \geq 1$, where set $a_0=0$. We prove by induction on $r$.

    For the base case, we show that $f(v_{1, b_0, k-1}) \geq f(v_{2,1,1})$. Suppose that $f(v_{1, b_0, k-1}) < f(v_{2,1,1})$. By the Relabeling rule ($a$), $v_{1, b_0, k-1}\in e_{1,b_0}=\left\{v_{0,1,1}, v_{1, b_0, 1}, \ldots, v_{1, b_0, k-1}\right\}$, and by the Relabeling rule  ($a'$), there exists an edge $e_{2,1}=\left\{v_{1,i_1,i_2}, v_{2, 1, 1}, \ldots, v_{2, 1, k-1}\right\}\in E(\G)$.
    
    If $v_{1,i_1,i_2}\in e_{1,b_0}$, then $i_1=b_0$. 
     Since $v_{0,1,1}\in V^\circ$,  we have $b_0>1$ and thus $v_{1,1,1}\notin e_{1,b_0}$. Since $f(v_{1,1,1})\geq f(v_{1, b_0, i_2})>0$, we also have $v_{1,1,1}\in V^\circ$. Thus there exists an edge $e_{2,j}=\{v_{1,1,1},v_{2,j,1},\ldots, v_{2,j,k-1}\}$ with $j\neq 1$ containing $v_{1,1,1}$. 
According to Case 1, $f(v_{1,1,1})\geq f(v_{1,b_0,i_2})$. By the Relabelling rules ($b'$), ($c'$) and ($d'$), $\sum_{p=1}^{k-1} f(v_{2, 1, p}) > \sum_{q=1}^{k-1} f(v_{2, j, q})$ (otherwise, $\sum_{p=1}^{k-1} f(v_{2, 1, p}) = \sum_{q=1}^{k-1} f(v_{2, j, q})$, then by ($c'$), we have $b_0\leq 1$, a contradiction).
Let $\mathcal{G'}$ be the new hypergraph obtained by the edge-switching  operation $e_{2,j} \underset{v_{1,b_0,i_2}}{\stackrel{v_{1,1,1}}{\rightleftharpoons}} e_{2,1}$. Then $\mathcal{G'}\in \mathcal{T}_\pi$ and according to  Lemma~\ref{switching}, we have $\lambda(\mathcal{G}')< \lambda(\mathcal{G})$, a contradiction.

Now assume $v_{1,i_1,i_2}\notin e_{1,b_0}$. If $\sum_{p=1}^{k-1} f(v_{2, 1, p}) > \sum_{q=1}^{k-1} f(v_{1, b_0, q})$, by the definition of $v_{0,1,1}$, $f(v_{1, i_1, i_2}) \leq f(v_{0,1,1})$. Let $\mathcal{G'}$ be the new hypergraph obtained by the edge-switching operation  $e_{2,1} \underset{v_{0,1,1}}{\stackrel{v_{1,i_1,i_2}}{\rightleftharpoons}} e_{1,b_0}$. Then $\mathcal{G'}\in \mathcal{T}_\pi$ and according to Lemma~\ref{switching}, we have $\lambda(\mathcal{G}')< \lambda(\mathcal{G})$. This leads to a contradiction.  Thus $\sum_{p=1}^{k-1} f(v_{2, 1, p}) \leq \sum_{q=1}^{k-1} f(v_{1, b_0, q})$. Recall the assumption that $f(v_{1, b_0, k-1}) < f(v_{2,1,1})$. 
Then $\sum_{p=2}^{k-1} f(v_{2, 1, p}) < \sum_{q=1}^{k-2} f(v_{1, b_0, q})$. Therefore, we have $f(v_{1, i_1, i_2})+\sum_{p=2}^{k-1} f(v_{2, 1, p}) < \sum_{q=1}^{k-2} f(v_{1, b_0, q})+f(v_{0,1,1})$. Let $\mathcal{G'}$ be the new hypergraph obtained by the edge-switching operation  $e_{2,1} \underset{v_{1,b_0,k-1}}{\stackrel{v_{2,1,1}}{\rightleftharpoons}} e_{1,b_0}$. Then $\mathcal{G'}\in \mathcal{T}_\pi$ and according to Lemma~\ref{switching},  $\lambda(\mathcal{G}')< \lambda(\mathcal{G})$. This leads to a contradiction again.  Therefore, $f(v_{1,b_0,k-1})\geq f(v_{2,1,1})$.

Now suppose  $f(v_{r, b_{r-1}-a_{r-1}, k-1}) \geq f(v_{r+1,1,1})$ for $r\ge 2$. 
    We show that $f(v_{r+1, b_{r}-a_{r}, k-1}) \geq f(v_{r+2,1,1})$. Suppose not, i.e., $f(v_{r+1, b_{r}-a_{r}, k-1}) < f(v_{r+2,1,1})$. By the Relabeling rule ($a'$), there exist an edge  $e_{r+1,b_r-a_r}=\left\{v_{r,i_1,i_2}, v_{r+1, b_r-a_r, 1}, \ldots, v_{r+1, b_r-a_r, k-1}\right\}$ with $v_{r+1, b_r-a_r, k-1}\in e_{r+1,b_r-a_r}$, and an edge $e_{r+2,1}=\left\{v_{r+1,j_1,j_2}, v_{r+2, 1, 1}, \ldots, v_{r+2, 1, k-1}\right\}$ with $v_{r+2, 1, 1}\in e_{r+2,1}$.
    
%

 If $v_{r+1,j_1,j_2}\in e_{r+1,b_r-a_r}$ then $j_1=b_r-a_r.$
If there exists a vertex $v_{r+1,l_1,l_2}$ $(l_1< b_r-a_r)$ and an edge $e_{r+2,j}\,(j\neq 1)$ such that $v_{r+1,l_1,l_2}\in e_{r+2,j}=\{v_{r+1,l_1,l_2},v_{r+2,j,1},\ldots, v_{r+2,j,k-1}\}$. 
From Case 1, $f(v_{r+1,l_1,l_2})\geq f(v_{r+1, b_r-a_r, j_2})$. By the Relabeling rule ($b'$), $\sum_{p=1}^{k-1} f(v_{r+2, 1, p}) \ge \sum_{q=1}^{k-1} f(v_{r+2, j, q})$. We claim that $\sum_{p=1}^{k-1} f(v_{r+2, 1, p}) > \sum_{q=1}^{k-1} f(v_{r+2, j, q})$. Otherwise, we 
 have $\sum_{p=1}^{k-1} f(v_{r+2, 1, p}) = \sum_{q=1}^{k-1} f(v_{r+2, j, q})$. Then by the Relabeling rule ($c'$), we have $b_r-a_r<l_1$, or  $b_r-a_r=l_1$ and $j_2<l_2$, each case contradicts to the assumption $l_1< b_r-a_r$.
Let $\mathcal{G'}$ be the new hypergraph obtained by the switching-operation $e_{r+2,j} \underset{v_{r+1, b_r-a_r, j_2}}{\stackrel{v_{r+1,l_1,l_2}}{\rightleftharpoons}} e_{r+2,1}$. Then $\mathcal{G'}\in \mathcal{T}_\pi$ and $\lambda(\mathcal{G}')< \lambda(\mathcal{G})$ by Lemma~\ref{switching}. This leads to a contradiction.
Now assume that such a vertex $v_{r+1,l_1,l_2}$ does not exist. Then either all the vertices in $V_{r+2}$ are adjacent to a vertex within the edge $e_{r+1,b_r-a_r}$, or there are no vertex in $V_{r+1}$ other than the vertices in $e_{r+1,b_r-a_r}$. 
For the former case, there exists  a boundary vertex $v_{r+1,x,y}\in V_{r+1}$ with $x\neq b_r-a_r$. Then we have a contradiction that  $0=f(v_{r+1, x, y})\geq f(v_{r+1, b_{r}-a_{r},j_2})>0$, where the first inequality holds according to Case 1, and $f(v_{r+1, b_{r}-a_{r},j_2})>0$ since $v_{r+1, b_{r}-a_{r},j_2}$ is an interior vertex as  $d(v_{r+1, b_{r}-a_{r},j_2})\ge 2$. For the latter case, we have  a boundary vertex $v_{r,x,y}\in V_{r}$ and thus we also have a contradiction that $0=f(v_{r,x,y})\geq f(v_{r, b_{r-1}-a_{r-1}, k-1}) \geq f(v_{r+1,1,1})\geq f(v_{r+1, b_{r}-a_{r},j_2})>0$.

Now assume $v_{r+1,j_1,j_2}\notin e_{r+1,b_r-a_r}.$ If $\sum_{p=1}^{k-1} f(v_{r+2, 1, p}) > \sum_{q=1}^{k-1} f(v_{r+1, b_r-a_r, q})$, by the induction hypothesis, $f(v_{r+1, j_1, j_2}) \leq f(v_{r,i_1,i_2})$. 
Let $\mathcal{G'}$ be the new hypergraph obtained by the edge-switching operation $e_{r+2,1} \underset{v_{r,i_1,i_2}}{\stackrel{v_{r+1,j_1,j_2}}{\rightleftharpoons}} e_{r+1,b_r-a_r}$. Then $\mathcal{G'}\in \mathcal{T}_\pi$ and according to Lemma~\ref{switching}, $\lambda(\mathcal{G}')< \lambda(\mathcal{G})$. This leads to a contradiction. 
Thus $\sum_{p=1}^{k-1} f(v_{r+2, 1, p}) \leq \sum_{q=1}^{k-1} f(v_{r+1, b_r-a_r, q})$. Recall the assumption that $f(v_{r+1, b_{r}-a_{r}, k-1}) < f(v_{r+2,1,1})$. Then $\sum_{p=2}^{k-1} f(v_{r+2, 1, p}) < \sum_{q=1}^{k-2} f(v_{r+1, b_r-a_r, q})$. Thus $f(v_{r+1,j_1,j_2})+\sum_{p=2}^{k-1} f(v_{r+2, 1, p}) < f(v_{r,i_1,i_2})+\sum_{q=1}^{k-2} f(v_{r+1, b_r-a_r, q})$. 
Let $\mathcal{G'}$ be the new hypergraph obtained by the edge-switching operation $e_{r+2,1} \underset{v_{r+1,b_r-a_r,k-1}}{\stackrel{v_{r+2,1,1}}{\rightleftharpoons}} e_{r+1,b_r-a_r}$. Then $\mathcal{G'}\in \mathcal{T}_\pi$ and according to Lemma~\ref{switching},  $\lambda(\mathcal{G}')< \lambda(\mathcal{G})$ . This leads to a contradiction too.  Therefore, we have $f(v_{r+1, b_{r}-a_{r}, k-1}) \geq f(v_{r+2,1,1})$.

\end{proof}

According to Lemma~\ref{f(v)<f(u)}, if $v_{s,i,p}\prec v_{t,j,q}$ then $f(v_{s,i,p})\geq f(v_{t,j,q})$. Subsequent lemmas  will demonstrate that under the same condition, $d(v_{s,i,p})\leq d(v_{t,j,q})$.

\begin{lem}\label{B2}
   Suppose $\left\{v_{s+1, i_1, p_1}, v_{s, i, p}\right\} \subset e_1 \in E(\G)$ and $\left\{v_{t+1, j_1, q_1}, v_{t, j, q}\right\} \subset e_2 \in E(\G)$. If $v_{s+1, i_1, p_1} \prec v_{t+1, j_1, q_1}$, then $v_{s, i, p} \prec v_{t, j, q}$.
	
\end{lem}
\begin{proof}
    Let $e_1=\{v_{s, i, p}, v_{s+1, i_1, 1},\ldots,v_{s+1, i_1, k-1}\}$ and $e_2=\{v_{t, j, q}, v_{t+1, j_1, 1},\ldots,v_{t+1, j_1, k-1}\}$. 
    Since $v_{s+1, i_1, p_1} \prec v_{t+1, j_1, q_1}$,  we consider $s<t$ and $s=t, i_1<j_1$ (when $s=t, i_1=j_1, p_1<q_1$,  we have $v_{s, i, p} = v_{t, j, q}$). 
      
 If $s<t$, then $v_{s, i, p} \prec v_{t, j, q}$. We are done.
 
 If $s=t$ and $i_1<j_1$, by the Relabeling rule ($b'$), $\sum_{p=1}^{k-1} f(v_{s+1, i_1, p}) \geq \sum_{q=1}^{k-1} f(v_{t+1, j_1, q})$.
If  $\sum_{p=1}^{k-1} f(v_{s+1, i_1, p})=\sum_{q=1}^{k-1} f(v_{t+1, j_1, q})$, then by the Relabeling rule ($c'$),  we have either $s<t$ (which is impossible), or $s=t$ and $i<j$. Therefore, $v_{s, i, p} \prec v_{t, j, q}$. We are done. Finally, assume $\sum_{p=1}^{k-1} f(v_{s+1, i_1, p}) > \sum_{q=1}^{k-1} f(v_{t+1, j_1, q})$. 
If $v_{t, j, q} \prec v_{s, i, p}$, then, according to Lemma~\ref{f(v)<f(u)}, $f(v_{t, j, q}) \geq f(v_{s, i, p})$.
	Let $\mathcal{G'}$ be the new hypergraph obtained by the edge-switching operation $e_1 \underset{v_{t, j, q}}{\stackrel{v_{s, i, p}}{\rightleftharpoons}} e_2$, then $\mathcal{G'}\in \mathcal{T}_\pi$ and according to Lemma~\ref{switching}, $\lambda(\mathcal{G}')< \lambda(\mathcal{G})$. This leads to a contradiction. Therefore, we have $v_{s, i, p} \prec v_{t, j, q}$.
	
\end{proof}


\begin{lem}\label{child}
    Every interior vertex $v\in V^{\circ}$ has a child $w$ with $f(w)<f(v)$.
\end{lem}
\begin{proof}
 If $v\neq v_{0,1,1}$ (i.e., $v$ is not the root vertex), 
assume $v\in V_s \,(s\geq 1)$,  then there exists $u\in V_{s-1}$ such that $v,u\in e=\{u,v,v_1,\ldots,v_{k-2}\}\in E(\G)$, where $v_i\in V_s$ for $1\le i\le k-2$.  
According to Lemma~\ref{f(v)<f(u)}, we have $f(u)\geq f(v)$, $f(v)\ge f(w)$ and $f(v_i)\ge f(w)$ for all $w\in N(v)\cap V_{s+1}$ and $1\le i\le k-2$.

Suppose, contrary to the claim, that all vertices $w\in N(v)\cap V_{s+1}$ satisfy $f(w)=f(v)$. Therefore, $f(v_i)\geq f(w)=f(v)$ for $1\le i\le k-2$. Thus
\begin{equation*}
	\begin{aligned}
		\lambda(\mathcal{G})f(v)&=L_{\mathcal{G}}f(v)=\frac{1}{k-1}\sum_{x\sim v}(f(v)-f(x))\\
		&=\frac{1}{k-1}(f(v)-f(u))+\sum_{w\in N(v)\cap V_{s+1}}(f(v)-f(w))+\sum_{i=1}^{k-2}(f(v)-f(v_i))\\
		&\leq\frac{1}{k-1}(f(v)-f(u))+\sum_{w\in N(v)\cap V_{s+1}}(f(v)-f(w))+\sum_{i=1}^{k-2}(f(v)-f(w))\\
		&=\frac{1}{k-1}(f(v)-f(u))\\
		&\leq 0.
	\end{aligned}
\end{equation*}
However, $f(v)>0$ and $\lambda(\mathcal{G})>0$, which leads to a  contradiction.

Now consider the case where $v= v_{0,1,1}$ (with $v$ being the root vertex). Then $f(v)\ge f(w)$ for every vertex  $w\in N(v)\cap V_1$. Assume, contrary to our expectations that $f(w)=f(v)$ for all children $w\,(\in N(v)\cap V_1)$ of $v$. Consequently, 
\begin{equation*}
	\begin{aligned}
		\lambda(\mathcal{G})f(v)&=L_{\mathcal{G}}f(v)=\frac{1}{k-1}\sum_{w\sim v}(f(v)-f(w))= 0.
	\end{aligned}
\end{equation*}
However, $f(v)>0$ and $\lambda(\mathcal{G})>0$, this leads to a contradiction once again.

\end{proof}

\begin{lem}\label{B5}
   For any $v_{s,i,p},v_{t,j,q}\in V^{\circ}$, if $v_{s,i,p}\prec v_{t,j,q}$ then $d(v_{s,i,p})\leq d(v_{t,j,q})$.
\end{lem}
\begin{proof}
Since $v_{s,i,p}\prec v_{t,j,q}$, we have $s < t$, or $s = t$ but $i < j$, or $s = t, i = j$ but $p < q$. 

\noindent\textbf{Case 1:} $s=t$, $i=j$ but $p<q$, or $s=t$ but $i<j$.

Assume to the contrary that $d(v_{s,i,p})> d(v_{s,j,q})$.  Since $v_{s,i,p}\prec v_{s,j,q}$, we have $f(v_{s,i,p})\geq f(v_{s,j,q})$ by Lemma~\ref{f(v)<f(u)}. Let $\delta=d(v_{s,i,p})-d(v_{s,j,q})$. By Lemma~\ref{child}, there exists an edge $e_{s+1,i_1}=\{v_{s,i,p},v_{s+1,i_1,1},\ldots,v_{s+1,i_1,k-1}\}$ containing a child $v_{s+1,i_1,j_1}$ of $v_{s,i,p}$ such that $f(v_{s+1,i_1,j_1})<f(v_{s,i,p})$. Choose other $\delta-1$ edges $e_{s+1,i_l}=\{v_{s,i,p},v_{s+1,i_l,1},\ldots,v_{s+1,i_l,k-1}\}(2\leq l\leq \delta)$ and let $F=\{e_{s+1,i_1}, e_{s+1,i_2},\ldots, e_{s+1, i_\delta}\}$. By Lemma~\ref{f(v)<f(u)}, $f(v_{s+1,i_l,j})\leq f(v_{s,j,q})\leq f(v_{s,i,p})$ for $1\leq l\leq \delta$ and $1\leq j \leq k-1$.
Let $\mathcal{G'}$ be the new hypergraph obtained by the shifting-operation $v_{s,i,p} \xrightarrow{F} v_{s,j,q}$. Then $\mathcal{G'}\in \mathcal{T}_\pi$ and $\lambda(\mathcal{G}')< \lambda(\mathcal{G})$ by Lemma~\ref{shifting} (here the strict inequality holds since $f(v_{s+1,i_1,j_1})<f(v_{s,i,p})$). This leads to a contradiction.

\noindent\textbf{Case 2:} $s<t$.

According to Case 1,  it sufficient to prove that $d(v_{s+1,b_s-a_s,k-1}) \leq d(v_{s+2,1,1})$ for $s \geq 0$. Assume to the contrary that $d(v_{s+1,b_s-a_s,k-1}) > d(v_{s+2,1,1})$. According to Lemma~\ref{f(v)<f(u)}, $f(v_{s+1,b_s-a_s,k-1})\geq f(v_{s+2,1,1})$. Let $\delta=d(v_{s+1,b_s-a_s,k-1}) - d(v_{s+2,1,1})$.
By Lemma~\ref{child}, there exists an edge $e_{s+2,i_1}=\{v_{s+1,b_s-a_s,k-1},v_{s+2,i_1,1},\ldots,v_{s+2,i_1,k-1}\}$ containing a child $v_{s+2,i_1,j_1}$ of $v_{s+1,b_s-a_s,k-1}$ such that $f(v_{s+2,i_1,j_1})<f(v_{s+1,b_s-a_s,k-1})$ where $i_1\neq 1,j_1\neq 1$. Choose other $\delta-1$ edges $e_{s+2,i_l}=\{v_{s+1,b_s-a_s,k-1},v_{s+2,i_l,1},\ldots,v_{s+2,i_l,k-1}\}(2\leq l\leq \delta)$ and let $F$ be the set of edges  $\{e_{s+2,i_l} : 1\le l\le \delta\}$. By Lemma~\ref{f(v)<f(u)}, $f(v_{s+2,i_l,j})\leq f(v_{s+2,1,1})\leq f(v_{s+1,b_s-a_s,k-1})$ for $1\leq l\leq \delta, 1\leq j \leq k-1$. Let $\mathcal{G'}$ be the new hypergraph obtained by shifting $v_{s+1,b_s-a_s,k-1} \xrightarrow{F} v_{s+2,i_l,j_l}$> Then we have $\mathcal{G'}\in \mathcal{T}_\pi$ and $\lambda(\mathcal{G}')< \lambda(\mathcal{G})$ by Lemma~\ref{shifting} too. This leads to a contradiction.

\end{proof}

\noindent\textbf{Proof of Theorem~\ref{mainthm}:} 
We first show that the ordering $\prec$ indeed satisfies the SLO-ordering rules.
The SLO-ordering rule (S1) is satisfied by the Relabeling rules and the definition of $\prec$. The subsequent rules (S2) and (S3) follow directly from Lemma~\ref{B2} and  Lemma~\ref{f(v)<f(u)}, respectively. The SLO-ordering rule (S4) is derived from the application of the Relabeling rules  ($a$) and ($a'$). Finally, the SLO-ordering rule (S5) is confirmed as true according to Lemma~\ref{B5}. 

It remains to show that $\G$ is determined to be unique up to isomorphism. Assume there is another $k$-uniform SLO-supertree $\G^\prime\in \T_\pi$ having the Faber-Krahn property. Denote $n_0$ as the number of interior vertices of supertrees in $\T_\pi$. Order the vertices of $\G^\prime$ with SLO-ordering  $w_0\prec w_1\prec \ldots\prec w_{n_0-1}\prec\ldots\prec w_{n-1}$. Then, according to the SLO-ordering rule (S5), we have $d(w_0)\leq \ldots \leq d(w_{n_0-1})$ and $d(w_{n_0})=\ldots=d(w_{n-1})=1$. Let $W_i=\left\{w \mid \dist\left(w, w_0\right)=i\right\}$, $a_i=\left|W_i\right|$, and $b_i=\sum_{w \in W_i} d(w)$ for $0 \leq i \leq h$. We now relabel the vertices of $\G^\prime$ according to the following rules. 

\begin{itemize}
    \item[(1)]  Label vertex $w_0$ as $w_{0,1,1}$, designating it as the root of the supertree $\G^\prime$.
    \item[(2)]   All vertices of $\G^\prime$ adjacent to $w_{0,1,1}$ are relabeled as $w_{1,1,1}, \ldots, w_{1,1, k-1}$, $w_{1,2,1}, \ldots, w_{1,2, k-1}$, $\ldots, w_{1, b_0, 1}, \ldots, w_{1, b_0, k-1}$, and these vertices satisfy the following rules:
    \begin{itemize}
        \item[($a$)] The set $e_{1, i}=\left\{w_{0,1,1}, w_{1, i, 1}, \ldots, w_{1, i, k-1}\right\}$ is an edge of $\G^\prime$ for $1 \leq i \leq b_0$.
        \item[($b$)] $w_{0, 1, 1}\prec w_{1, i, p}$ for $1 \leq i \leq b_0, 1\leq p\leq k-1$.
        \item[($c$)] $w_{1, i, p}\prec w_{1, j, q}$ if $1 \leq i<j \leq b_0$ for $1\leq p,q\leq k-1$, or $1 \leq i=j \leq b_0$ and $1 \leq p<q \leq k-1$. 
    \end{itemize}
    \item[(3)] Assuming that all vertices of $W_r$ have been relabeled for $1 \leq r \leq h-1$. We relabel the vertices of $W_{r+1}$ as $w_{r+1,1,1}, w_{r+1,1,2}, \ldots, w_{r+1,1, k-1}$, $w_{r+1,2,1}, \ldots, w_{r+1,2, k-1}$, and $w_{r+1, b_r-a_r, 1}, \ldots, w_{r+1, b_r-a_r, k-1}$. These vertices satisfy the following rules:
    \begin{itemize}
        \item[($a^{\prime}$)] There exists an edge $e \in E(\G^\prime)$ such that $\left\{w_{r+1, i, 1}, w_{r+1, i, 2}, \ldots, w_{r+1, i, k-1}\right\} \subset e$ for $1 \leq i \leq b_r-a_r$.
        \item[($b^{\prime}$)] $w_{r, i, p}\prec w_{r+1, j, q}$ for all $1 \leq i \leq b_{r-1}-a_{r-1}, 1 \leq j \leq b_r-a_r, 1 \leq p, q \leq k-1$.
        \item[($c^{\prime}$)] $w_{r+1, i, p} \prec w_{r+1, j, q}$ if $1 \leq i<j \leq b_r-a_r$ for $1 \leq p, q \leq k-1$, or  $1 \leq i=j \leq b_r-a_r$ and $1 \leq p<q \leq k-1$.
    \end{itemize}
\end{itemize}
From the relabeling rules of the vertices in $\G^\prime$, we have 
$$
w_{s, i, p} \prec w_{t, j, q} \text { if and only if } (s,i,p)<(t,j,q) \text { in the lexicographic order}.
$$ 
Therefore, we have a natural isomorphic mapping $\phi$ from $\G$ to $\G^\prime$:  $\phi(v_{i,j,p})=w_{i,j,p}$, i.e. $\G\cong \G^\prime$. 

\qed

\section{Proof of Theorem~\ref{mainthm2}}
A sequence $\pi=\left(d_0, d_1, \ldots, d_{n_0-1},\ldots,d_{n-1}\right)$  is defined as a degree sequence of some supertree with $n$ vertices and $n_0$ interior vertices  if the first $n_0$ elements of the sequence correspond to the degrees of interior vertices and are arranged in non-decreasing order.
If we exchange two consecutive instances of $\pi$ simultaneously, with each altered value differs from its original counterpart by exactly one, then we obtain the set $\{d_0, d_1, \ldots,d_p-1, d_{p+1}+1, \ldots, d_{n_0-1},\ldots,d_{n-1}\}$ for some $0\le p\le n_0-2$. Clearly, when $d_p\ge 3$, the set $\{d_0, d_1, \ldots,d_p-1, d_{p+1}+1, \ldots, d_{n_0-1},\ldots,d_{n-1}\}$ corresponds to a degree sequence $\pi'=(d_0', d_1', \ldots, d_{n_0-1}',\ldots,d_{n-1}')$ (the order of $d_0, d_1, \ldots,d_p-1, d_{p+1}+1, \ldots, d_{n_0-1}$ would be rearranged in $\pi'$, if any). 
We call such an operation as a {\em unit transformation} on $d_p$ of $\pi=\left(d_0, d_1, \ldots, d_{n_0-1},\ldots,d_{n-1}\right)$. 
 We say that a sequence $\pi=\left(d_0, d_1, \ldots, d_{n_0-1},\ldots,d_{n-1}\right)$ is \textit{majorized} by sequence $\pi^{\prime}=\left(d_0^{\prime}, d_1^{\prime}, \ldots, d_{n_0-1}^{\prime},\ldots,d_{n-1}^{\prime}\right)$, denoted as $\pi \triangleleft \pi^{\prime}$, if $\sum_{i=0}^j d_i \leq \sum_{i=0}^j d_i^{\prime}$ for all $0\leq j\leq n-2$, and $\sum_{i=0}^{n-1} d_i=\sum_{i=0}^{n-1} d_i^{\prime}$.

 As a direct corollary from Lemmas~\ref{child} and \ref{f(v)<f(u)}, we have the following observation.
\begin{cor}\label{COR:lastchild}
Suppose $\G$  is a $k$-uniform supertree with an SLO-ordering $v_0\prec v_1\prec \ldots\prec v_{n-1}$ of its vertices. 
Then for any interior vertex $v_i$ and its last child of $v_l$ in this SLO-ordering, we have $f(v_l)<f(v_i)$.    
\end{cor} 

\begin{lem}\label{majorizate}
    Let $\pi=(d_0, d_1, \ldots, d_{n_0-1},\ldots,d_{n-1})$ be a degree sequence of $k$-uniform supertrees with $n$ vertices  and $n_0$ interior vertices. Suppose $d_p\geq 3$ for some $0\leq p \leq n_0-2$. Let $\pi'=(d_0', d_1', \ldots, d_{n_0-1}',\ldots,d_{n-1}')$ be the new degree sequence obtained from $\pi$ through a unit transformation on $d_p$.
If $\G=(V^{\circ}\cup\partial V, E^{\circ}\cup\partial E)$ and $\G'=\left(V'^{\circ }\cup\partial V', E'^{\circ}\cup\partial E'\right)$ have the Faber-Krahn property in $\T_\pi$ and $\T_{\pi'}$, respectively, then $\lambda(\G')<\lambda(\G)$.
\end{lem}
\begin{proof}
Let $f$ be the first Dirichlet eigenfunction of $\G$.  According to Theorem~\ref{mainthm},  the vertices of $\G$ have an SLO-ordering $v_0\prec v_1\prec \ldots\prec v_{n-1}$. According to (S5) of  SLO-ordering,  we have $d_\G(v_0)\leq \ldots \leq d_\G(v_{n_0-1})$ and $d_\G(v_{n_0})=\ldots= d_{\G}(v_{n-1})=1$, i.e.,  $d_{\G}(v_i)=d_i$ for $i=0,1, \ldots, n-1$.
Since $v_p\prec v_{p+1}$, $v_p$ is not a child of $v_{p+1}$ and $f(v_{p+1})\le f(v_p)$ according to Lemma~\ref{f(v)<f(u)}. 
Assume $v_p\in V_h=\left\{v \mid \dist_{\G}(v, v_0)=h\right\}$. By Lemma~\ref{child},  there exists a child $w\in V_{h+1}$ of $v_p$ with $f(w)<f(v_p)$. 
In fact, we claim that there exists an edge $e_1=\{v_p, u_1, \ldots, u_{k-1}\}$ with $v_{p+1}\notin e_1$, $f(u_{k-1})<f(v_p)$ and $f(u_j)\leq f(v_{p+1})\leq f(v_p)$ for $1\leq j \leq k-1$. 
Note that $d_{\G}(v_p) \geq 3$. We  take the edge $\{v_p, u_1,\ldots,u_{k-1}\}$ containing the last child $u_{k-1}$ of $v_p$ in this SLO-ordering as $e_1$. 
Then $u_1, u_2, \ldots, u_{k-1}\in V_{h+1}$.
By Corollary~\ref{COR:lastchild}, we have $f(u_{k-1})<f(v_p)$.
If $v_{p+1}\in V_h$, then clearly, $v_{p+1}\notin e_1$ and $f(u_j)\leq f(v_{p+1})\leq f(v_p)$ for $1\leq j \leq k-1$.
Now assume $v_{p+1}\in V_{h+1}$. Since $v_p\prec v_{p+1}$ are consecutive in this SLO-ordering, $v_{p+1}$ must be the first vertex in $V_{h+1}$ in this SLO-ordering of $\G$, i.e. $v_{p+1}\prec w$ for any $w\in V_{h+1}\backslash \{v_{p+1}\}$. By Lemma~\ref{f(v)<f(u)}, $f(v_{p+1})\geq f(w)$ for all the vertices $w\in V_{h+1}\backslash \{v_{p+1}\}$. Thus it is sufficient to show that $v_{p+1}\notin e_1$.
If $v_{p+1}\in e_1$, then $d_{\G}(v_p)\le 2$,  a contradiction to the assumption $d_{\G}(v_p)\ge 3$. We are done.
According to the claim,  we have an edge $e_1=\{v_p, u_1, \ldots, u_{k-1}\}$ with $v_{p+1}\notin e_1$, $f(u_{k-1})<f(v_p)$ and $f(u_j)\leq f(v_{p+1})\leq f(v_p)$ for $1\leq j \leq k-1$. Moreover, $\{v_{p+1}, u_1, \ldots, u_{k-1}\}\notin E(\G)$ since $\G$ is a supertree.
Let $\G_0$ be the supertree resulting from the shifting-operation $v_p \xrightarrow{e_1} v_{p+1}$ on $\G$, then $d_{\G'}(v_p)=d_p-1$ and $d_{\G'}(v_{p+1})=d_{p+1}+1$, and the degrees of the  other vertices remains unchanged. Hence $\mathcal{G}_{0}\in \mathcal{T}_{\pi'}$ and $\lambda(\mathcal{G}_0)< \lambda(\mathcal{G})$ by Lemma~\ref{shifting}. 
Therefore,  $\lambda(\G')\leq \lambda(\G_0) <\lambda(\G)$.

\end{proof}

\begin{lem}[\cite{Marshall Majorization}]\label{seqtransform}
    If $\pi$ and $\pi^{\prime}$ are two  integer sequences with $\pi^{\prime} \triangleleft \pi$, then $\pi^{\prime}$ can be obtained from $\pi$ by a finite sequence of  unit transformations. 
\end{lem}

From Lemma \ref{majorizate} and Lemma \ref{seqtransform}, we have the following corollary.

\begin{cor}\label{majorizate2}
    Let $\pi=(d_0, d_1, \ldots, d_{n_0-1},\ldots,d_{n-1})$ and $\pi'=(d_0', d_1', \ldots, d_{n_0-1}',\ldots,d_{n-1}')$ be two degree sequences of some supertrees with the same number of vertices $n$ and interior vertices $n_0$. If $\pi' \triangleleft \pi$ and $\pi\not=\pi'$, then $\lambda(\G')<\lambda(\G)$.
\end{cor}

Finally, we give the proof of Theorem~\ref{mainthm2}.
\begin{proof}[Proof of Theorem~\ref{mainthm2}:] 
Let $\mathcal{G}=(V^{\circ}\cup\partial V, E^{\circ}\cup\partial E)$ be a $k$-uniform supertree with Faber-Krahn property in $\T_d^{n,n_0}$.
Let $\pi=\left(d_0, d_1, \ldots, d_{n_0-1}, 1, \ldots, 1\right)$ be the degree sequence of $\G$. Then $d \leqslant d_0 \leqslant d_1 \leqslant \cdots \leqslant d_{n_0-1}$. Let $\pi^{\prime}=\left(d, \ldots, d, d^{\prime}, 1, \ldots, 1\right)$, where $d^{\prime}=\sum_{i=0}^{n_0-1}d_i-(n_0-1)d$. 
If $\pi\neq\pi^{\prime}$, then $\pi' \triangleleft \pi$. Let $\G^{\prime}$ be a supertree with Faber-Krahn property in $\T_{\pi'}$. By Corollary~\ref{majorizate2}, $\lambda(\G^{\prime})<\lambda\left(\G\right)$, a contradiction to the Faber-Krahn property of $\G$. The necessity is proved. The sufficiency follows from the uniqueness of the SLO-supertree in $\mathcal{T}_{\pi}$.
\end{proof}

\section{Remarks and discussions}
In this article, we introduce a spiral-like ordering (SLO-ordering) of vertices for supertrees and prove that the SLO-supertree has the Faber-Krahn property among all supertrees with a given degree sequence. Furthermore, among degree sequences that have a minimum degree $d$ for interior vertices, the SLO-supertree with degree sequence $(d,\ldots,d, d', 1, \dots, 1)$ possesses  the Faber-Krahn property.
These results answer Problem~\ref{PROB: p1} for supertrees with boundary. It is very interesting to explore this problem for other families of hypergraphs, for example, unicyclic hypergraphs, under other given "volume"?

\vspace{5pt}
\noindent{\bf Acknowledgements}:
This work was supported by the National Key Research and Development Program of China (2023YFA1010203), the National Natural Science Foundation of China (No.12471336, 12071453), and the Innovation Program for Quantum Science and Technology (2021ZD0302902).

\end{document}